\documentclass[12 pt, reqno]{amsart}
\usepackage{amsmath,times,epsfig,amssymb,amsbsy,amscd,amsfonts,amstext,graphicx,color,bm,amsthm}
\usepackage[all]{xy}
\usepackage{graphicx}
\usepackage{hyperref}

\usepackage{tikz,subfigure}
\usepackage{tikz-cd}

\newcommand{\KK}{\mathbb{K}}
\newcommand{\NN}{\mathbb{N}}
\newcommand{\RR}{\mathbb{R}}

\newcommand{\A}{\mathcal{A}}

\newcommand \ideal[1] {\langle #1 \rangle}

\newtheorem{Theorem}{Theorem}[section]
\newtheorem{Definition}[Theorem]{Definition}

\newtheorem{Lemma}[Theorem]{Lemma}
\newtheorem{Proposition}[Theorem]{Proposition}
\newtheorem{Corollary}[Theorem]{Corollary}
\newtheorem{Remark}[Theorem]{Remark}
\newtheorem{Example}[Theorem]{Example}
\newtheorem{Conjecture}[Theorem]{Conjecture}

\DeclareMathOperator{\Der}{Der}

\DeclareMathOperator{\pdeg}{pdeg}

\DeclareMathOperator{\gin}{gin}
\DeclareMathOperator{\rgin}{rgin}
\DeclareMathOperator{\LT}{LT}
\DeclareMathOperator{\GL}{GL}
\DeclareMathOperator{\HF}{HF}
\DeclareMathOperator{\DegRevLex}{DegRevLex}

\DeclareMathOperator{\reg}{reg}
\newcommand \M{{\mathcal M}}

\begin{document}

\title[$k$-Lefschetz properties]{$k$-Lefschetz properties, sectional matrices and hyperplane arrangements}

\begin{abstract} 
In this article, we study the $k$-Lefschetz properties for non-Artinian algebras, proving that several known results in the Artinian case can be generalized in this setting. Moreover, we describe how to characterize the graded algebras having the $k$-Lefschetz properties using sectional matrices. We then apply the obtained results to the study of the Jacobian algebra of hyperplane arrangements, with particular attention to the class of free arrangements.
\end{abstract}

\author{Elisa Palezzato}
\address{Elisa Palezzato, Department of Mathematics, Hokkaido University, Kita 10, Nishi 8, Kita-Ku, Sapporo 060-0810, Japan.}
\email{palezzato@math.sci.hokudai.ac.jp}
\author{Michele Torielli}
\address{Michele Torielli, Department of Mathematics, GI-CoRE GSB, Hokkaido University, Kita 10, Nishi 8, Kita-Ku, Sapporo 060-0810, Japan.}
\email{torielli@math.sci.hokudai.ac.jp}


\date{\today}
\maketitle



\section{Introduction}

In \cite{harima2009generic}, the authors introduced the notions of weak and strong $k$-Lefschetz properties as a generalization of the weak and strong Lefschetz properties. These concepts are connected to many topics in algebraic geometry, commutative algebra and combinatorics. Some of these connections are quite surprising and there are still several open questions. See for example \cite{harima2013lefschetz} and \cite{migliore2013survey}.

The goal of this paper is to continue and extend the study of the Lefschetz properties for non-Artinian algebras the authors started in \cite{palezzato2019lefschetz}. In order to do that, we will use the sectional matrix \cite{BPT2016}, a matrix that encodes the Hilbert function of successive hyperplane sections of a homogeneous ideal. In particular, we will highlight how the sectional matrix of a graded algebra plays an important role in the study of Lefschetz properties.  We will then apply the obtained results to the study of the Jacobian algebra of a hyperplane arrangement, with particular attention to the class of free arrangements, as a natural step after \cite{palezzato2019lefschetz}.

This paper is organized as follow. In Section 2, we recall the notions of weak and strong Lefschetz properties with their basic attributes and we characterize such properties via Hilbert functions. In Section 3, we introduce the notions of $k$-WLP and $k$-SLP and describe their basic properties. In Section 4, we recall the notion of almost revlex ideal and we put it in connection with the $k$-Lefschetz properties. In Section 5, we connect the non-Artinian case to the Artinian one. In Section 6, we recall the notion of sectional matrix and we describe how to characterize algebras having $k$-Lefschetz properties using such matrix.  In Section 7, we recall the definitions and basic properties of hyperplane arrangements. In Section 8, we analyze the Jacobian algebra of an arrangement from the $k$-Lefschetz properties point of view, with particular attention to the class of free arrangements.


\section{Lefschetz properties}


Throughout this paper, we will consider $\KK$ a field of characteristic $0$ and $S=\KK[x_1,\dots,x_l]$ the polynomial ring with standard grading. 

\begin{Definition}\label{def:stronstableideal}
A monomial ideal $I$ of $S$ is said to be  \textbf{strongly stable} if for every power-product $t \in I$ and
every $i,j$ such that $1\le i<j\le l$ and $x_j|t$, the power-product $x_i\cdot t/x_j\in I$.
\end{Definition}

\begin{Example}\label{ex:nonss-ssideal}
The ideal $I=\ideal{x^3, x^2y, xy^2, xyz}$ is not strongly stable in $\mathbb{R}[x,y,z]$ because $x\cdot xyz/y = x^2z\not\in I$. It is enough to add $x^2z$ as a minimal generator to $I$ to obtain a strongly stable ideal.
\end{Example}

\begin{Remark}\label{rem:ss=borelfixed} Since $\KK$ has characteristic $0$, an ideal is strongly stable if and only if it is fixed under the natural action of the Borel subgroup of $\GL(l)$.
\end{Remark}

\begin{Definition} 
Let $\sigma$ be a term ordering on $S$ and $f$ a non-zero polynomial in $S$. 
Then $\LT_\sigma(f)= \max_\sigma\{{\rm Supp}(f)\}$, where ${\rm Supp}(f)$ is the set of all power-products 
appearing with non-zero coefficient in $f$.  If
$I$ is an ideal in $S$, then the 
\textbf{leading term ideal} (or \textit{initial ideal})
of $I$ is the ideal $\LT_\sigma(I)$ of $S$ generated by
 $\{\LT_\sigma(f) \mid f \in I{\setminus}\{0\}\;\}$.
\end{Definition}

The following theorem is due to Galligo \cite{galligo1974propos}.

\begin{Theorem}[\cite{galligo1974propos}]\label{thm:Galligo}
Let $I$ be a homogeneous ideal of $S$, with $\sigma$ a term ordering such that
$x_1>_\sigma x_2 >_\sigma \dots >_\sigma x_l$. Then there exists a Zariski
open set $U\subseteq\GL(l)$ and a strongly stable ideal $J$ such that for each  $g\in U$, $\LT_\sigma(g(I)) = J$.
\end{Theorem}

\begin{Definition} 
The strongly stable ideal $J$ given in Theorem~\ref{thm:Galligo} is called the \textbf{generic initial ideal} with respect to $\sigma$ of $I$ and it is denoted by {\boldmath$\gin_\sigma(I)$}.  
In particular, when $\sigma= \DegRevLex$, $\gin_\sigma(I)$ is simply denoted by {\boldmath$\rgin(I)$}.
\end{Definition}

As described in \cite{BPT2016}, we can read a lot of information on an ideal from its generic initial ideal. For example, we have the following.
\begin{Remark}\label{rem:rginsameHFasideal} 
Let $I$ be a homogeneous ideal of $S$. Then the Hilbert function of $S/I$ coincides with the one of $S/\rgin(I)$.
\end{Remark}


We can now introduce the notions of weak and strong Lefschetz properties for graded algebras.
\begin{Definition}\label{def:WLPandSLP} 
Let $R$ be a graded ring over $\KK$, and $R = \bigoplus_{i\geq 0} R_i$ its decomposition into homogeneous
components with $\dim_\KK (R_i) < \infty$.
\begin{enumerate}
\item The graded ring $R$ is said to have the \textbf{weak Lefschetz property (WLP)}, 
if there exists an element $\ell \in R_1$ such that
the multiplication map
\begin{align*}
\times \ell \colon R_i &\rightarrow R_{i+1}\\ 
f &\mapsto \ell f
\end{align*}
 is full-rank for every $i \geq 0$. In this case, $\ell$ is called a \textbf{weak Lefschetz element}.
\item The graded ring $R$ is said to have the \textbf{strong Lefschetz property (SLP)}, 
if there exists an element $\ell \in R_1$ such that
the multiplication map
\begin{align*}
\times \ell^s \colon R_i &\rightarrow R_{i+s}\\ 
f &\mapsto \ell^sf
\end{align*}
is full-rank for every $i \geq 0$ and $s\geq 1$. In this case, $\ell$ is called a \textbf{strong Lefschetz element}.
\end{enumerate}
\end{Definition}

In \cite{palezzato2019lefschetz}, the authors studied in depth these two notions. Between all the results obtained, we state here the two that play an important role in this paper. 


\begin{Lemma}\cite[Lemma 2.8]{palezzato2019lefschetz}\label{lemma:SSWLPx_l}
Let $I$ be a strongly stable ideal of $S$.
Then $S/I$ has the SLP (respectively the WLP) if and only if $~S/I$ has the SLP (respectively the WLP) with Lefschetz element $x_l$. 
\end{Lemma}


\begin{Proposition}\cite[Proposition 2.9]{palezzato2019lefschetz}\label{prop:ILPginLP}
Let $I$ be a homogeneous ideal of $S$.
Then the graded ring $S/I$ has the SLP (respectively the WLP) if and only if $S/\rgin(I)$ 
has the SLP (respectively the WLP).
\end{Proposition}

%

Similarly to Remark 6.11 of \cite{harima2013lefschetz}, to check if a quotient algebra has the SLP, it is enough to check the differences of its Hilbert function.

\begin{Proposition}\label{prop:SLPdiffHF} Let $I$ be a homogeneous ideal of $S$.
Then the graded ring $S/I$ has the SLP with strong Lefschetz element $\ell$ if and only if for every $s\ge1$ and $d\ge0$ we have that
\begin{equation}\label{eq:SLPHilbertdiffequal}
\HF(S/(I+\ideal{\ell^s}),d)=\max\{\HF(S/I,d)-\HF(S/I,d-s),0\},
\end{equation}
where $\HF(S/I,r)=0$ for all $r<0$. 
\end{Proposition}
\begin{proof} By slightly generalizing the arguments of the proof of Lemma~1.2 from \cite{conca2003reduction} (using that $\LT_{\DegRevLex}(gI+\ideal{x^k_l})= \LT_{\DegRevLex}(gI)+\ideal{x^k_l}$ for all $k \geq 1$), one obtains that the Hilbert function of $S/(\rgin(I)+\ideal{x_l^s})$ is equal to the Hilbert function of $S/(I+\ideal{\ell^s})$ for a general linear form $\ell \in S_1$ and all $s \geq 1$. This fact, together with Lemma~\ref{lemma:SSWLPx_l} and Proposition~\ref{prop:ILPginLP}, implies that it is enough to prove the statement when $I$ is a strongly stable ideal and $\ell=x_l$.

Assume that $S/I$ has the SLP with strong Lefschetz element $x_l$. Fix $s\ge1$ and $d\ge0$. By assumption the multiplication map
$$\times x_l^s \colon (S/I)_{d-s} \rightarrow (S/I)_d $$
has full rank. If this map is surjective, then $\HF(S/I,d)-\HF(S/I,d-s)\le0$ and all the generators of $(S/I)_d$ are divisible by $x_l^s$. This implies that $S/(I+\ideal{x_l^s})=0$ and hence $\HF(S/(I+\ideal{x_l^s}),d)=0=\max\{\HF(S/I,d)-\HF(S/I,d-s),0\}$. On the other hand, if the multiplication map is injective, then $\HF(S/I,d)-\HF(S/I,d-s)\ge0$ and for every power-product $x_l^st\in I_d$, we have that $t\in I_{d-s}$. This implies that $I_d$ is the union of the two disjoint spaces $x_l^s\cdot I_{d-s}$ and the space generated by all power products $u\in I_d$ such that $x_l^s\nmid u$.
As a consequence we have that $\HF(S/(I+\ideal{x_l^s}),d)=\HF(S/I,d)-\HF(S/I,d-s)=\max\{\HF(S/I,d)-\HF(S/I,d-s),0\}$. This implies that $S/I$ satisfies \eqref{eq:SLPHilbertdiffequal}.

Assume now that $S/I$ satisfies  \eqref{eq:SLPHilbertdiffequal}. 
Fix $s\ge1$ and $d\ge0$. If we have that $\HF(S/(I+\ideal{x_l^s}),d)=\HF(S/I,d)-\HF(S/I,d-s)>0$, then the multiplication map
$$\times x_l^s \colon (S/I)_{d-s} \rightarrow (S/I)_d $$
cannot be surjective. Suppose that such map is also not injective. This implies that there exists a power-product $t\in S_{d-s}\setminus I_{d-s}$ such that $x_l^st\in I_d$. This implies that $I_d$ contains strictly the union of the two disjoint spaces $x_l^s\cdot I_{d-s}$ and the space generated by all power products $u\in I_d$ such that $x_l^s\nmid u$. As a consequence, $\HF(S/(I+\ideal{x_l^s}),d)+\HF(S/I,d-s)>\HF(S/I,d)$, but this is a contraddiction, and hence the multiplication map has full rank.
On the other hand, if $\HF(S/(I+\ideal{x_l^s}),d)=0$, then $(S/(I+\ideal{x_l^s}))_d=0$ and $\HF(S/I,d)-\HF(S/I,d-s)\le0$. If $(S/I)_d=0$, then the multiplication map 
$$\times x_l^s \colon (S/I)_{d-s} \rightarrow (S/I)_d $$
is clearly surjective. If $(S/I)_d\ne0$, since $(S/(I+\ideal{x_l^s}))_d=0$, then all the generators of $(S/I)_d$ are divisible by $x_l^s$. This implies that also in this case the multiplication map is surjective, and hence it has full rank.

Since this argument works for all $s\ge1$ and $d\ge0$, then $S/I$ has the SLP with strong Lefschetz element $x_l$.
\end{proof}

If we consider only the case when $s=1$ in the argument of Proposition~\ref{prop:SLPdiffHF}, we obtain a characterization of the WLP via the first differences of the Hilbert function.

\begin{Proposition}\label{prop:WLPdiffHF} Let $I$ be a homogeneous ideal of $S$.
Then the graded ring $S/I$ has the WLP with weak Lefschetz element $\ell$ if and only if for every $d\ge0$ we have that
$$\HF(S/(I+\ideal{\ell}),d)=\max\{\HF(S/I,d)-\HF(S/I,d-1),0\},$$
where $\HF(S/I,-1)=0$. 
\end{Proposition}

\section{$k$-WLP and $k$-SLP}

As a generalization of the Lefschetz properties of Definition~\ref{def:WLPandSLP}, we can introduce the $k$-WLP and $k$-SLP. See \cite{harima2013lefschetz} and \cite{harima2009generic} for more details.
 
\begin{Definition}\label{def:klefsprop} Let $R$ be a graded ring over $\KK$, $R = \bigoplus_{i\geq 0} R_i$ its decomposition into homogeneous
components with $\dim_\KK (R_i) < \infty$, and $k$ a positive integer. The graded ring $R$ is said to have the {\boldmath$k$}\textbf{-SLP} (respectively the {\boldmath$k$}\textbf{-WLP}) if there exist linear elements $\ell_1,\dots,\ell_k\in R_1$ satisfying the following two conditions
\begin{enumerate}
\item $R$ has the SLP (respectively WLP) with Lefschetz element $\ell_1$,
\item $R/\ideal{\ell_1,\dots,\ell_{i-1}}$ has the SLP (respectively WLP) with Lefschetz element $\ell_i$, for all $i=2,\dots,k$.
\end{enumerate}
In this case we will say that $(R,\ell_1,\dots,\ell_k)$ has the $k$-SLP (respectively $k$-WLP).
\end{Definition}
\begin{Remark} As noted in Remark 6.2 of \cite{harima2013lefschetz}, if $(R,\ell_1,\dots,\ell_k)$ has the $k$-SLP (respectively $k$-WLP), then $\ell_1$ is a Lefschetz element for $R$. However, if $g_1$ is another Lefschetz element for $R$, there do not necessarily exist $g_2,\dots,g_k\in R_1$ such that $(R,g_1,\dots,g_k)$ has the $k$-SLP (respectively $k$-WLP)
\end{Remark}

Similarly to Proposition 6.9 of \cite{harima2013lefschetz}, we have the following two statements.

\begin{Proposition}\label{prop:inductionkwlp} Let $I$ be a strongly stable ideal of $S$ and $1\le k\le l$. $S/I$ has the $k$-SLP (respectively the $k$-WLP) if and only if $S/I$ has the SLP (respectively the WLP) with Lefschetz element $x_l$ and $S/(I+\ideal{x_l})$ has the $(k{-}1)$-SLP (respectively the $(k{-}1)$-WLP).
\end{Proposition}
\begin{proof} Clearly, if $S/I$ has the SLP (respectively the WLP) with Lefschetz element $x_l$ and $S/(I+\ideal{x_l})$ has the $(k{-}1)$-SLP (respectively the $(k{-}1)$-WLP), then $S/I$ has the $k$-SLP (respectively the $k$-WLP).

Assume that $(S/I,\ell_1,\dots,\ell_k)$ has the $k$-SLP (respectively the $k$-WLP). By Remark~\ref{rem:ss=borelfixed}, the subgroup $\mathcal{H}\subset\GL(l)$ composed of all the matrices of the form
\begin{center}
$$\begin{pmatrix}
I_{l-1} & a  \\
0 & b \\
\end{pmatrix}$$
\end{center}
where $a\in \KK^{l-1}$ and $b\in\KK\setminus\{0\}$ stabilizes any strongly stable ideal.
Since the elements $\ell_1,\dots,\ell_k$ are generic, we can assume that $\ell_1=\sum_{j=1}^l\alpha_jx_j$ and $\alpha_l\ne0$.
This implies that there exists $\varphi\in \mathcal{H}$ such that $\varphi(\ell_1)=x_l$. By applying $\varphi$ to $(S/I,\ell_1,\dots,\ell_k)$,
we obtain that $(S/\varphi(I),\varphi(\ell_1),\dots,\varphi(\ell_k))=(S/I,x_l,\varphi(\ell_2),\dots,\varphi(\ell_k))$ has the $k$-SLP (respectively the $k$-WLP), and hence we obtain the claimed equivalence.
\end{proof}

\begin{Proposition}\label{prop:kwlplastvarok} Let $I$ be a strongly stable ideal of $S$ and $1\le k\le l$. $S/I$ has the $k$-SLP (respectively the $k$-WLP) if and only if $(S/I, x_l,\dots,x_{l-k+1})$ has the $k$-SLP (respectively $k$-WLP).
\end{Proposition}
\begin{proof} Clearly, if $(S/I, x_l,\dots,x_{l-k+1})$ has the $k$-SLP (respectively $k$-WLP), then $S/I$ has the $k$-SLP (respectively the $k$-WLP).

Assume that $S/I$ has the $k$-SLP (respectively the $k$-WLP). By Proposition~\ref{prop:inductionkwlp}, we can assume that $(S/I,x_l,\ell_2,\dots,\ell_k)$ has the $k$-SLP (respectively $k$-WLP). Let $\bar{S}=\KK[x_1,\dots,x_{l-1}]$ and $\bar{I}=\bar{S}\cap I$. Then $S/(I+\ideal{x_l})\cong \bar{S}/\bar{I}$ and $\bar{I}$ is a strongly stable ideal of $\bar{S}$. Therefore by induction on $l$, using Proposition~\ref{prop:inductionkwlp}, we obtain that we can take $\ell_2=x_{l-1},\dots, \ell_k=x_{l-k+1}$.
\end{proof}

\begin{Example} Let $I=\ideal{x^2,xy,xz}$ be a strongly stable ideal of $S=\mathbb{R}[x,y,z,w]$. Since I has no minimal generators divisible by $w$, $(S/I,w)$ has the $1$-WLP. The quotient $S/(I+\ideal{w})$ has an increasing Hilbert function. However, the multiplication map $\times z\colon (S/(I+\ideal{w}))_1\to(S/(I+\ideal{w}))_2$ is not injective, and hence $z$ is not a Lefschetz element for $S/(I+\ideal{w})$. By Proposition~\ref{prop:kwlplastvarok}, $S/I$ does not have the $2$-WLP.
\end{Example}

Similarly to Proposition 6.15 of \cite{harima2013lefschetz}, we can generalize Proposition~\ref{prop:ILPginLP} and reduce the study of $k$-Lefschetz properties to the strongly stable case.

\begin{Theorem}\label{theo:kSLPtorgin} Let $I$ be a homogeneous ideal of $S$ and $1\le k\le l$. Then the following two conditions are equivalent
\begin{enumerate}
\item $S/I$ has the $k$-SLP (respectively the $k$-WLP),
\item $(S/\rgin(I), x_l,\dots,x_{l-k+1})$ has the $k$-SLP (respectively the $k$-WLP).
\end{enumerate}
\end{Theorem}
\begin{proof} We first show that the two conditions are equivalent for the $k$-WLP. Let $1\le j\le k$. Lemma~1.2 of \cite{conca2003reduction} shows that the Hilbert function of $S/(\rgin(I)+\ideal{x_l, \dots, x_{l-j+1}})$ is equal to the
Hilbert function of $S/(I+\ideal{\ell_1,\dots, \ell_j})$ for a general linear form $\ell_1,\dots, \ell_j \in S_1$. This fact, together with Propositions~\ref{prop:WLPdiffHF}, \ref{prop:inductionkwlp} and \ref{prop:kwlplastvarok}, gives us the equivalence between the two conditions for the $k$-WLP.

We now show the equivalence for the $k$-SLP. Let $1\le j\le k$. Similarly to the proof of Proposition~\ref{prop:SLPdiffHF}, by modifying the proof of Lemma~1.2 of \cite{conca2003reduction}, we obtain  the equality between the Hilbert function of $S/(\rgin(I)+\ideal{x_l, \dots, x_{l-j+2}, x_{l-j+1}^s})$ and the
Hilbert function of $S/(I+\ideal{\ell_1,\dots, \ell_{j-1}, \ell_j^s})$ for general linear forms $\ell_1,\dots, \ell_j \in S_1$ and $s\ge1$. This fact, together with Propositions~\ref{prop:SLPdiffHF}, \ref{prop:inductionkwlp} and \ref{prop:kwlplastvarok}, gives us the equivalence between the two conditions for the $k$-SLP.
\end{proof}


\section{Almost revlex ideals and Lefschetz properties}

In this section, we recall the notion of almost revlex ideal, a special class of monomial ideals, and we put it in connection with the $k$-Lefschetz properties.

\begin{Definition}\label{def:ARL} 
A monomial ideal $I$ of $S$ is called an \textbf{almost revlex ideal}, if for any power-product $t$ in the minimal generating set of $I$, every other power-product $t'$ of $S$ with $\deg(t')=\deg(t)$ and $t'>_{\DegRevLex}t$ belongs to the ideal $I$.
\end{Definition}

\begin{Remark}\label{rem:almostimpliesstrongstable}
Every almost revlex ideal is strongly stable.
\end{Remark}

In general, not all strongly stable ideals are almost revlex ideals.
\begin{Example}
Consider the ideal $I=\ideal{x^3, x^2y, xy^2, xyz}$ in $\mathbb{R}[x,y,z]$ of Example~\ref{ex:nonss-ssideal}. As seen before, it is not strongly stable, and hence it is not almost revlex. On the other hand,  also the strongly stable ideal $J=I+\ideal{x^2z}$ is not almost revlex.
In fact, $xyz\in J$, but $y^3\notin J$. If we consider the ideal $J+\ideal{y^3}$, finally, this is an almost revlex ideal. 
\end{Example}

\begin{Remark}\label{rem:ARLsameHFsameideal} 
If two almost revlex ideals have the same Hilbert function, then they coincide.
\end{Remark}

If we assume that $S=\KK[x,y]$, then all strongly stable ideals are almost revlex ideals.
\begin{Lemma}\label{lemma:ssimpliesalmostrevlex} Let $I$ be a strongly stable ideal of $S=\KK[x,y]$. Then $I$ is an almost revlex ideal and it is uniquely determined by the Hilbert function.
\end{Lemma}
\begin{proof} Since $S=\KK[x,y]$, we have that if $t$ and $t'$ are two power-product in $S$ such that $\deg(t')=\deg(t)$ and $t'>_{\DegRevLex}t$, then $t'=x^\alpha t /y^\alpha$ for some $\alpha\ge0$. By Definitions~\ref{def:stronstableideal} and \ref{def:ARL}, this clearly implies that every strongly stable ideal is an almost revlex ideal.
 
Finally, $I$ is determined only by the Hilbert function, by Remark~\ref{rem:ARLsameHFsameideal}.
\end{proof}


Almost revlex ideals have several interesting properties, as described in \cite{harima2013lefschetz}, \cite{bertone2019almost} and \cite{palezzato2019lefschetz}. The following result add the $l$-SLP to the list.

\begin{Theorem}\label{theo:almrevlexlSLP} Let $I$ be an almost revlex ideal of $S$. Then $(S/I, x_l,\dots, x_1)$ has the $l$-SLP.
\end{Theorem}
\begin{proof} By Corollary~5.7 from \cite{palezzato2019lefschetz}, $S/I$ has the SLP with Lefschetz element $x_l$. Since $I$ is an almost revlex ideal, then $I\cap\KK[x_1,\dots,x_{l-1}]$ is an almost revlex ideal. Since every almost revlex ideal is strongly stable by Remark~\ref{rem:almostimpliesstrongstable}, we conclude by Proposition~\ref{prop:inductionkwlp} and induction on $l$.
\end{proof}


Similarly to Proposition 3.15 of \cite{harima2013lefschetz}, if $l=2$, then $S/I$ has always the $2$-SLP.
\begin{Theorem}\label{theo:SLP2dim} Let $I$ be a homogeneous ideal of $S=\KK[x,y]$. Then $S/I$ has the $2$-SLP.
\end{Theorem}
\begin{proof} By Proposition~\ref{prop:ILPginLP}, it is enough to prove the statement when $I$ is a strongly stable ideal. By Lemma~\ref{lemma:ssimpliesalmostrevlex}, if $I$ is a strongly stable ideal, then it is an almost revlex ideal, and hence by Theorem~\ref{theo:almrevlexlSLP}, $S/I$ has the $2$-SLP.
\end{proof}

\begin{Remark} As noted in Remark~3.3 from \cite{harima2013lefschetz}, Theorem~\ref{theo:SLP2dim} is false if we do not assume that $S$ has standard grading.
\end{Remark}

\begin{Corollary}\label{cor:l-2slptolslp} Let $I$ be a homogeneous ideal of $S$. Then the following facts are equivalent
\begin{enumerate}
\item $S/I$ has the $l$-SLP (respectively the $l$-WLP),
\item $S/I$ has the $(l{-}1)$-SLP (respectively the $(l{-}1)$-WLP),
\item $S/I$ has the $(l{-}2)$-SLP (respectively the $(l{-}2)$-WLP).
\end{enumerate}
\end{Corollary}
\begin{proof} Clearly (1) implies (2) and (2) implies (3). On the other hand, by the definition of $k$-Lefschetz properties and Theorem~\ref{theo:SLP2dim}, we get that (3) implies (1).
\end{proof}

\section{Regularity and $k$-Lefschetz properties}

To connect the Artianian case and non-Artinian one, a key role is played by the regularity of a homogeneous ideal. 

\begin{Definition}\label{def:regularity}
Let $I$ be a homogeneous ideal of $S$. 
The \textbf{Castelnuovo-Mumford regularity} of $I$, denoted~{\boldmath$\reg(I)$},
is the maximum of the numbers $d_i-i$,
where  $d_i = \max\{j\mid \beta_{i,j}(I)\ne0\}$ and $\beta_{i,j}(I)$ are the graded  
Betti numbers of $I$.
\end{Definition}

In \cite{bayer1987criterion}, the authors described the connection between the Castelnuovo-Mumford regularity of an ideal and the maximal degree of the minimal generators of its generic initial ideal.
\begin{Theorem}[\cite{bayer1987criterion}]\label{theo:regginequalideal} 
Let $I$ be a homogeneous ideal of $S$. Then $\reg(I)=\reg(\rgin(I))$. 
Moreover, if $I$ is a strongly stable ideal, then $\reg(I)$ is the highest degree of a minimal generator of $I$.
\end{Theorem}
\begin{Remark}\label{rem:regandminimalgen} If $I$ is a homogeneous ideal of $S$, then the highest degree of a minimal generator of $I$ is smaller or equal to $\reg(I)$.
\end{Remark}

Given $I$ a homogeneous ideal of $S$, we will denote by {\boldmath$\hat{I}$} the following ideal
$$\hat{I}= I+\ideal{x_1,\dots,x_l}^{\reg(I)+1}.$$

In Corollary~5.4 from \cite{palezzato2019lefschetz}, the authors described how to reduce to the Artinian case if we are interested in studying the WLP. We can generalize such result for the $k$-SLP.

\begin{Theorem}\label{theo:kSLPtoArtinian}
Let $I$ be a homogeneous ideal of $S$ and $1\le k\le l$. Then the following facts are equivalent
\begin{enumerate}
\item the graded ring $S/I$ has the $k$-SLP,
\item the graded Artinian ring $S/\hat{I}$ has the $k$-SLP.
\end{enumerate}
\end{Theorem}
\begin{proof} Assume that $(S/I,\ell_1,\dots,\ell_k)$ has the $k$-SLP, and let $0\le j\le k-1$. By construction $(S/(I+\ideal{\ell_1,\dots,\ell_{j}}))_d=(S/(\hat{I}+\ideal{\ell_1,\dots,\ell_{j}}))_d$, for any $0\le d\le \reg(I)$, where if $j=0$, the ideal $\ideal{\ell_1,\dots,\ell_j}=\ideal{0}$.
Since $(S/I,\ell_1,\dots,\ell_k)$ has the $k$-SLP, the multiplication map 
$$\times \ell^s_{j+1} \colon (S/(\hat{I}+\ideal{\ell_1,\dots,\ell_j}))_{d} \rightarrow (S/(\hat{I}+\ideal{\ell_1,\dots,\ell_j}))_{d+s}$$ 
has full-rank every time $d-s \le \reg(I)$.

On the other hand, $(S/(\hat{I}+\ideal{\ell_1,\dots,\ell_{j}}))_d=0$ for any $d\ge\reg(I)+1$. This implies that the multiplication map 
$$\times \ell^s_{j+1} \colon (S/(\hat{I}+\ideal{\ell_1,\dots,\ell_j}))_{d} \rightarrow (S/(\hat{I}+\ideal{\ell_1,\dots,\ell_j}))_{d+s}$$
is always surjective when $d-s \ge \reg(I)+1$, and hence, $(S/\hat{I},\ell_1,\dots,\ell_k)$ has the $k$-SLP.

Assume now that $S/\hat{I}$ has the $k$-SLP, and let $0\le j\le k-1$. By Theorem~\ref{theo:kSLPtorgin}, $(S/\rgin(\hat{I}),x_l,\dots, x_{l-k+1})$ has the $k$-SLP. By Theorem~\ref{theo:regginequalideal}, $\rgin(I)$ has no minimal generators of degree greater or equal to $\reg(I)+1$. Hence $\rgin(\hat{I})=\rgin(I)+\ideal{x_1,\dots,x_l}^{\reg(I)+1}$. This implies that, similarly to the previous part, $(S/(\rgin(I)+\ideal{x_l,\dots,x_{l-j+1}}))_d=(S/(\rgin(\hat{I})+\ideal{x_l,\dots,x_{l-j+1}}))_d$, for any $0\le d\le \reg(I)$, where if $j=0$, the ideal $\ideal{x_l,\dots,x_{l-j+1}}=\ideal{0}$. Consider $s\ge1$, then the multiplication map $\times x_{l-j}^s$ from $(S/(\rgin(I)+\ideal{x_l,\dots,x_{l-j+1}}))_{d}$ to
$(S/(\rgin(I)+\ideal{x_l,\dots,x_{l-j+1}}))_{d+s}$
has full-rank every time $d+s \le \reg(I)$. On the other hand, since $\rgin(I)$ has no minimal generators of degree greater or equal to $\reg(I)+1$, the multiplication map $\times x_{l-j}^s$
is injective for every $d \ge \reg(I)$. This implies that, if $s=1$, all the multiplication maps by $x_{l-j}$ have full rank, and hence that $(S/\rgin(I),x_l,\dots, x_{l-k+1})$ has the $k$-WLP. 
Let $s\ge2$ and $d<\reg(I)<d+s$. Consider the multiplication map $\times x_{l-j}^s$ from $(S/(\rgin(I)+\ideal{x_l,\dots,x_{l-j+1}}))_{d}$ to
$(S/(\rgin(I)+\ideal{x_l,\dots,x_{l-j+1}}))_{d+s}$.
This map can be written as the composition of the multiplication maps $\times x_{l-j}^{\reg(I)-d}$ from $(S/(\rgin(I)+\ideal{x_l,\dots,x_{l-j+1}}))_{d}$ to $(S/(\rgin(I)+\ideal{x_l,\dots,x_{l-j+1}}))_{\reg(I)}$ and $\times x_{l-j}^{d+s-\reg(I)}$ from $(S/(\rgin(I)+\ideal{x_l,\dots,x_{l-j+1}}))_{\reg(I)}$ to $(S/(\rgin(I)+\ideal{x_l,\dots,x_{l-j+1}}))_{d+s}$. Notice that both maps have full-rank.
By \cite[Proposition~2.10]{palezzato2019lefschetz}, the Hilbert function of $S/(\rgin(I)+\ideal{x_l,\dots,x_{l-j+1}})$ is unimodal, and, since $\rgin(I)$ has no minimal generators of degree greater or equal to $\reg(I)+1$, such Hilbert function is increasing for every $d\ge\reg(I)$.
This implies that we have to analyze only the following two cases.
If 
\begin{align*}
&\HF(S/(\rgin(I)+\ideal{x_l,\dots,x_{l-j+1}}),d)\\
\le&\HF(S/(\rgin(I)+\ideal{x_l,\dots,x_{l-j+1}}),\reg(I))\\
\le&\HF(S/(\rgin(I)+\ideal{x_l,\dots,x_{l-j+1}}),d+s),
\end{align*}
then both multiplication maps $\times x_{l-j}^{\reg(I)-d}$ and $\times x_{l-j}^{d+s-\reg(I)}$ are injective and hence so is $\times x_{l-j}^{s}$.
If 
\begin{align*}
&\HF(S/(\rgin(I)+\ideal{x_l,\dots,x_{l-j+1}}),d)\\
>&\HF(S/(\rgin(I)+\ideal{x_l,\dots,x_{l-j+1}}),\reg(I))\\
=&\HF(S/(\rgin(I)+\ideal{x_l,\dots,x_{l-j+1}}),d+s),
\end{align*}
 then the multiplication map $\times x_{l-j}^{\reg(I)-d}$ is surjective and $\times x_{l-j}^{d+s-\reg(I)}$ is an injective map between spaces of the same dimension and hence it is also surjective. This implies that the multiplication map $\times x_{l-j}^{s}$ is surjective.
This proves that $(S/\rgin(I),x_l,\dots, x_{l-k+1})$ has the $k$-SLP. By Theorem~\ref{theo:kSLPtorgin}, $S/I$ has the $k$-SLP.
\end{proof}

If we consider only the case when $s=1$ in the argument of Theorem~\ref{theo:kSLPtoArtinian}, we can describe how to reduce to the Artinian case if we are interested in studying the $k$-WLP.

\begin{Corollary}\label{cor:kWLPtoArtinian}
Let $I$ be a homogeneous ideal of $S$ and $1\le k\le l$. Then the following facts are equivalent
\begin{enumerate}
\item the graded ring $S/I$ has the $k$-WLP,
\item the graded Artinian ring $S/\hat{I}$ has the $k$-WLP.
\end{enumerate}
\end{Corollary}

Similarly to Theorem~\ref{theo:kSLPtoArtinian}, also the study of almost revlex ideals can be reduced to the Artinian case.

\begin{Theorem}\label{theo:almostrevlexArtinian} Let $I$ be a monomial ideal of $S$. Then $I$ is an almost revlex ideal if and only if $\hat{I}$ is an almost revlex ideal.
\end{Theorem}
\begin{proof} Assume that $I$ is an almost revlex ideal. Let $t$ be a minimal generator of $\hat{I}$ and $t'$ a power-product such that $\deg(t)=\deg(t')$ and $t'>_{\DegRevLex}t$. If $\deg(t)\le\reg(I)$, then $t$ is also a minimal generator of $I$ and hence $t'\in I\subseteq\hat{I}$.
If $\deg(t)=\reg(I)+1$ then, by construction, $t'\in\hat{I}$.
Since $\hat{I}$ has no minimal generators of degree higher than $\reg(I)+1$, this implies that $\hat{I}$ is an almost revlex ideal.

Assume now that $\hat{I}$ is an almost revlex ideal. Let $t$ be a minimal generator of $I$ and $t'$ a power-product such that $\deg(t)=\deg(t')$ and $t'>_{\DegRevLex}t$. By Remark~\ref{rem:regandminimalgen}, $\deg(t)\le\reg(I)$. This implies that $t$ is a minimal generator of $\hat{I}$ and that $t'\in\hat{I}$.
Since $I_d=\hat{I}_d$ for all $0\le d\le\reg(I)$, then $t'\in I$. This proves that $I$ is an almost revlex ideal.
\end{proof}

\section{Sectional matrix and $k$-Lefschetz properties}

It seems natural to investigate the connections between the sectional matrix that encodes the Hilbert function of successive hyperplane sections of a graded algebra and the $k$-Lefschetz properties of such algebra.

In this section, we recall the definition and basic properties of the sectional matrix for the quotient algebra $S/I$, as described in \cite{BPT2016}. We then describe how to determine if a graded algebra has the $k$-SLP or $k$-WLP by looking at its sectional matrix.

\begin{Definition}\label{def:SM} 
Let $I$ be a homogeneous ideal of $S$. The \textbf{sectional matrix} 
of $S/I$ is the function $\{1,\dots,l\}\times \NN \longrightarrow \NN$
$$\M_{S/I}(i,d) = \HF(S/(I+\ideal{\ell_1,\dots,\ell_{l-i}}),d),$$
where $\ell_1,\dots,\ell_{l-i}$ are generic linear forms.
Notice that $$\M_{S/I}(l,d) = \HF(S/I, d).$$
\end{Definition}

\begin{Remark} As described in Theorem~4.1 of \cite{BPT2016}, even if the sectional matrix has an infinite numbers of columns, 
to describe the matrix it is enough the knowledge of the first $\reg(I)$ columns.
\end{Remark}

The following result reduces the study of the sectional matrix of
a homogeneous ideal to the combinatorial behaviour of a monomial ideal.

\begin{Theorem}\cite[Lemma~3.8]{BPT2016}\label{theo:rgin}
Let $I$ be a homogeneous ideal of $S$. 
Then 
$$\M_{S/I} (i,d) = \M_{S/\rgin(I)}(i,d) =
\HF(S/(\rgin(I)+\ideal{x_l,\dots,x_{i+1}}),d),$$
where if $i=l$, then $\ideal{x_l,\dots,x_{i+1}}=\ideal{0}$.
\end{Theorem}

\begin{Example}\label{ex:sectionalmat} Let $I=\ideal{x^2,xy,y^2,xz}$ be an ideal of $S=\mathbb{R}[x,y,z]$. Then the sectional matrix of $S/I$ is given by
$$
\begin{array}{rccccccc} 
          &  _0 & _1 & _2 & _3  & _4  & \dots\\ 
\M_{S/I}(1,d):&  1 & 1 & 0 & 0 & 0 & \dots\\ 
\M_{S/I}(2,d):&  1 & 2 & 0 & 0 & 0 & \dots\\
\HF(S/I,d)=\M_{S/I}(3,d):&  1 & 3 & 2 & 2 & 2 & \dots
\end{array} 
$$
\end{Example}

There are several known results that connect the algebraic properties of an ideal, the entries of the sectional matrix and the shape of the associated generic initial ideal. The most important for this article is the following.
\begin{Theorem}\cite[Theorem~6.6]{Gin-freearr}\label{theo:sectmatrixgenxk}
Let $I$ be a non-zero homogeneous ideal of $S$, $2\le i\le l$ and $d\ge1$.
Then $$\M_{S/I}(i,d)-\M_{S/I}(i,d-1)\le \M_{S/I}(i-1,d).$$
Moreover, the equality holds if and only if $\rgin(I)$ has no minimal generator of degree~$d$ divisible by $x_i$.
\end{Theorem}

Using the language of sectional matrices, we can rephrase Proposition~\ref{prop:WLPdiffHF} and characterize the graded algebras having the WLP via sectional matrices.
\begin{Proposition}\label{prop:WLPdiffSectMat} Let $I$ be a homogeneous ideal of $S$.
Then the graded ring $S/I$ has the WLP if and only if for every $0\le d\le \reg(I)$ we have that
\begin{equation}\label{eq:WLPsectionalmat}
\M_{S/I}(l-1,d)=\max\{\M_{S/I}(l,d)-\M_{S/I}(l,d-1),0\},
\end{equation}
where $\M_{S/I}(l,-1)=0$. 
\end{Proposition}
\begin{proof} By Definition~\ref{def:SM} and Proposition~\ref{prop:WLPdiffHF}, we just need to show that, if for every $0\le d\le \reg(I)$,
the sectional matrix of $S/I$ satisfies \eqref{eq:WLPsectionalmat}, then $S/I$ has the WLP. 
By Theorem~\ref{theo:regginequalideal}, $\rgin(I)$ has no minimal generators of degree greater or equal to $\reg(I)+1$. This implies that by Theorem~\ref{theo:sectmatrixgenxk}, that $\M_{S/I}(l,d)-\M_{S/I}(l,d-1)= \M_{S/I}(l-1,d)$ for every $d\ge\reg(I)+1$. Hence, the sectional matrix of $S/I$ satisfies \eqref{eq:WLPsectionalmat}, for every $d\ge0$. We conclude by Proposition~\ref{prop:WLPdiffHF}.
\end{proof}

Similarly to Proposition~\ref{prop:WLPdiffSectMat}, using Proposition~\ref{prop:inductionkwlp}, we obtain the following generalization of Proposition~\ref{prop:WLPdiffHF} for the $k$-WLP.

\begin{Theorem}\label{theo:kWLPdiffSectMat} Let $I$ be a homogeneous ideal of $S$ and $1\le k\le l$.
Then the graded ring $S/I$ has the $k$-WLP if and only if for every $0\le d\le \reg(I)$ and $0\le j\le k-1$ we have that
\begin{equation}\label{eq:kWLPsectionalmat}
\M_{S/I}(l-j-1,d)=\max\{\M_{S/I}(l-j,d)-\M_{S/I}(l-j,d-1),0\},
\end{equation}
where $\M_{S/I}(l,-1)=0$. 
\end{Theorem}

\begin{Example}\label{ex:sectionalmatkWLP} Let $I$ be the ideal of Example~\ref{ex:sectionalmat}. Then $\reg(I)=2$ and the sectional matrix of $S/I$ satisfies condition
\eqref{eq:kWLPsectionalmat}, and hence $S/I$ has the $3$-WLP.
\end{Example}

\begin{Corollary}\label{corol:triantokwlp} Let $I$ be a homogeneous ideal of $S$. If there exists $0\le k\le l-1$ such that
$$\M_{S/I}(l-k-1,d)=\M_{S/I}(l-k,d)-\M_{S/I}(l-k,d-1),$$
for all $d\ge0$, then $S/I$ has the $(k{+}1)$-WLP.
\end{Corollary}
\begin{proof} By Theorem~\ref{theo:sectmatrixgenxk}, the existence of such index $k$ implies that $\rgin(I)$ has no minimal generator divisible by $x_{l-k}$. Since $\rgin(I)$ is a strongly stable ideal, then $\rgin(I)$ has no minimal generator divisible by $x_{r}$ for any $l-k\le r\le l$. Again by Theorem~\ref{theo:sectmatrixgenxk}, this implies that $\M_{S/I}(l-j-1,d)=\M_{S/I}(l-j,d)-\M_{S/I}(l-j,d-1)$ for any $0\le j\le k$. We conclude by Theorem~\ref{theo:kWLPdiffSectMat}.
\end{proof}
In general the statement of Corollary~\ref{corol:triantokwlp} is not an equivalence.

\begin{Example} Let $I$ be the ideal of Example~\ref{ex:sectionalmat}. As seen in Example~\ref{ex:sectionalmatkWLP}, $S/I$ has the $3$-WLP. However, $\M_{S/I}(2,2)>\M_{S/I}(3,2)-\M_{S/I}(3,1)$.
\end{Example}

Similarly to Theorem~\ref{theo:kWLPdiffSectMat}, we can generalize Proposition~\ref{prop:SLPdiffHF} using the language of sectional matrices and characterize the graded algebras having the $k$-SLP via such matrices.

\begin{Theorem}\label{theo:kSLPdiffSectMat} Let $I$ be a homogeneous ideal of $S$ and $1\le k\le l$.
$(S/I,\ell_1,\dots,\ell_k)$ has the $k$-SLP if and only if for every $d\ge0$, $0\le j\le k-1$ and $s\ge1$ we have that
\begin{align*}
&\HF(S/(I+\ideal{\ell_1,\dots,\ell_{j-1},\ell_j^s}),d)=\\
=&\HF(S/(\rgin(I)+\ideal{x_l,\dots,x_{l-j},x_{l-j+1}^s}),d)=\\
=&\max\{\M_{S/I}(l-j,d)-\M_{S/I}(l-j,d-s),0\},
\end{align*}
where $\M_{S/I}(l,r)=0$ for all $r\le0$. 
\end{Theorem}
\begin{proof} The first equality holds independently from the $k$-SLP. In fact, as described in the proof of Theorem~\ref{theo:kSLPtorgin}, by modifying the proof of Lemma~1.2 of \cite{conca2003reduction}, we obtain the equality between the Hilbert function of $S/(\rgin(I)+\ideal{x_l, \dots, x_{l-j+2}, x_{l-j+1}^s})$ and the
Hilbert function of $S/(I+\ideal{\ell_1,\dots, \ell_{j-1}, \ell_j^s})$ for general linear forms $\ell_1,\dots, \ell_j \in S_1$ and $s\ge1$.
The equivalence between the $k$-SLP and the second equality is a consequence of Propositions~\ref{prop:inductionkwlp} and \ref{prop:SLPdiffHF}, and Theorem~\ref{theo:kSLPtorgin}.
\end{proof}

For a graded algebra, having the $k$-SLP implies that the last $k$ rows of the sectional matrix are unimodal functions.

\begin{Proposition}\label{prop:kSLPshaperows} Let $I$ be a homogeneous ideal of $S$ and $1\le k\le l$. If $S/I$ has the $k$-SLP, then $\M_{S/I}(l-j,-)$ is an unimodal function, for all $j=0,\dots, k-1$.
\end{Proposition}
\begin{proof} By Theorem~\ref{theo:kSLPtorgin}, $(S/\rgin(I), x_l,\dots,x_{l-k+1})$ has the $k$-SLP. By Theorem~\ref{theo:rgin}, $\M_{S/I} (i,d) {=} \M_{S/\rgin(I)}(i,d) {=} \HF(S/(\rgin(I){+}\ideal{x_l,\dots,x_{i+1}}),d),$ for all $1\le i\le l$ and $d\ge0$. This implies that the statement follows by combining Proposition~\ref{prop:inductionkwlp} with Proposition~2.12 from \cite{palezzato2019lefschetz}.
\end{proof}

\begin{Remark}\label{rem:sectnonarttoart} The sectional matrix of $S/\hat{I}$ can be easily obtained from the one of $S/I$. In particular, for every $1\le i\le l$, we have that
$$\M_{S/\hat{I}}(i,d)=
\begin{cases}
\M_{S/I}(i,d)&\text{ if $~0\le d\le\reg(I)$}\\
0&\text{ if $~d\ge\reg(I)+1.$}\\
\end{cases}
$$
\end{Remark}

\begin{Definition} Let $h=(h_0,\dots, h_c)$ be a unimodal sequence of positive integers and $h_i$ the maximum of $h$. Then $h$ is said to be \textbf{quasi-symmetric} if, for every $i<j\le c$, $h_j$ coincides with one of $\{h_0,\dots, h_i\}$.
\end{Definition} 

In Corollary~5.11 from \cite{palezzato2019lefschetz}, the authors proved that if we assume that $l=3$ and that $S/I$ has the SLP, then $\rgin(I)$ is an almost revlex ideal. Similarly to Theorem~6.29 from \cite{harima2013lefschetz}, we can generalize this result to any dimension. 

\begin{Theorem}\label{theo:stronstimpliealmostrevlex} Let $I$ be a homogeneous ideal of $S$ such that $S/I$ has the $l$-SLP.
Suppose that $(\M_{S/I}(l-j,0),\dots,\M_{S/I}(l-j,\reg(I)))$ is quasi-symmetric for all $0\le j\le l-4$. Then $\rgin(I)$ is an almost revlex ideal and it is uniquely determined by the Hilbert function.
\end{Theorem}
\begin{proof} By Theorem~\ref{theo:kSLPtoArtinian}, $S/\hat{I}$ has the $l$-SLP, and hence, by Theorem~\ref{theo:kWLPdiffSectMat}, $\M_{S/\hat{I}}(l-j,-)$ coincides with the $j$-th difference of the Hilbert function of $S/\hat{I}$. By Remark~\ref{rem:sectnonarttoart}, this implies that we are in the hypothesis of  \cite[Theorem~6.29]{harima2013lefschetz}, and hence, $\rgin(\hat{I})$ is an almost revlex ideal. By Theorem~\ref{theo:regginequalideal}, $\rgin(I)$ has no minimal generators of degree greater or equal to $\reg(I)+1$. Hence $\rgin(\hat{I})=\rgin(I)+\ideal{x_1,\dots,x_l}^{\reg(I)+1}$. This implies that $\rgin(\hat{I})=\widehat{\rgin(I)}$.
By Theorem~\ref{theo:almostrevlexArtinian}, $\rgin(I)$ is an almost revlex ideal.

Finally, $\rgin(I)$ is determined only by the Hilbert function, by Remark~\ref{rem:ARLsameHFsameideal}.
\end{proof}


\section{Preliminares on hyperplane arrangements}\label{sec:arr}


A finite set of affine hyperplanes $\A =\{H_1, \dots, H_n\}$ in $\KK^l$ 
is called a \textbf{hyperplane arrangement}. For each hyperplane $H_i$ we fix a defining linear polynomial $\alpha_i\in S$ such that $H_i = \alpha_i^{-1}(0)$, 
and let $Q(\A)=\prod_{i=1}^n\alpha_i$. An arrangement $\A$ is called \textbf{central} if each $H_i$ contains the origin of $\KK^l$. 
In this case, each $\alpha_i\in S$ is a linear homogeneous polynomial, and hence $Q(\A)$ is homogeneous of degree $n$. 


We denote by $\Der_{\KK^l} =\{\sum_{i=1}^l f_i\partial_{x_i}~|~f_i\in S\}$ the $S$-module of \textbf{polynomial vector fields} on $\KK^l$ (or $S$-derivations). 
Let $\delta =  \sum_{i=1}^l f_i\partial_{x_i}\in \Der_{\KK^l}$. Then $\delta$ is  said to be \textbf{homogeneous of polynomial degree} $d$ if $f_1, \dots, f_l$ are homogeneous polynomials of degree~$d$ in $S$. 
In this case, we write $\pdeg(\delta) = d$.

\begin{Definition} 
Let $\A$ be a central arrangement in $\KK^l$. Define the \textbf{module of vector fields logarithmic tangent} to $\A$ (or logarithmic vector fields) by
$$D(\A) = \{\delta\in \Der_{\KK^l}~|~ \delta(\alpha_i) \in \ideal{\alpha_i} S, \forall i\}.$$
\end{Definition}

The module $D(\A)$ is a graded $S$-module and we have that $$D(\A)= \{\delta\in \Der_{\KK^l}~|~ \delta(Q(\A)) \in \ideal{Q(\A)} S\}.$$

\begin{Definition} 
A central arrangement $\A$ in $\KK^l$ is said to be \textbf{free with exponents $(e_1,\dots,e_l)$} 
if and only if $D(\A)$ is a free $S$-module and there exists a basis $\delta_1,\dots,\delta_l \in D(\A)$ 
such that $\pdeg(\delta_i) = e_i$, or equivalently $D(\A)\cong\bigoplus_{i=1}^lS(-e_i)$.
\end{Definition}

%

Given an arrangement $\A$ in $\KK^l$, the \textbf{Jacobian ideal} $J(\A)$ of $\A$
is the ideal of $S$ generated by $Q(\A)$ and all its partial derivatives.

The Jacobian ideal has a central role in the study of free arrangements.
In fact, we can characterize freeness by looking at $S/J(\A)$ via the Terao's criterion.
Notice that Terao described this result for characteristic $0$, but the statement holds true for any characteristic as shown in \cite{palezzato2018free}.

\begin{Theorem}[\cite{terao1980arrangementsI}]\label{theo:freCMcod2} 
A central arrangement $\A$ in $\KK^l$ is free if and only if $S/J(\A)$ is $0$ or $(l{-}2)$-dimensional Cohen--Macaulay.
\end{Theorem}

In \cite{Gin-freearr}, the authors connected the study of generic initial ideals to the one of arrangements, obtaining a new characterization of freeness via the generic initial ideal of the Jacobian ideal.

\begin{Proposition}[\cite{Gin-freearr}]\label{prop:shapergin}
Let $\A =\{H_1, \dots, H_n\}$ be a central arrangement in $\KK^l$. 
Then $\rgin(J(\A))$ coincides with $S$ or 
its minimal generators include $x_1^{n-1}$, some positive power of
$x_2$, and no monomials only in $x_3,\dots, x_l$. 
\end{Proposition}

\begin{Example}\label{ex:NotFreeButLP}
Let $\A$ be the arrangement in $\RR^3$ with defining polynomial $Q(\A)=xyz(x+y+z)$.
In this case $\rgin(J(\A))=\ideal{x^3,x^2y,xy^2,y^4,y^3z}$.
\end{Example}

\begin{Theorem}[\cite{Gin-freearr}]\label{theo:firstequivfregin}
Let $\A =\{H_1, \dots, H_n\}$ be a central arrangement in $\KK^l$. 
Then $\A$ is free if and only if 
$\rgin(J(\A))$ coincides with $S$ or it is minimally generated by
$$x_1^{n-1},\; x_1^{n-2}x_2^{\lambda_1},\; \dots,\; x_2^{\lambda_{n-1}}$$ 
with $1\le\lambda_1<\lambda_2<\cdots<\lambda_{n-1}$ and $\lambda_{i{+}1}-\lambda_i= 1$ or $2$.
\end{Theorem}

\begin{Example}
Let $\A$ be the central arrangement in $\RR^3$ with defining polynomial $Q(\A)=xyz(x-y)(x-z)(y-z)$.
$\A$ is a free arrangement with exponents $(1,2,3)$.
In this case $\rgin(J(\A))=\ideal{x^5,x^4y,x^3y^2,x^2y^4,xy^5,y^7}$.
\end{Example}

\begin{Example}\label{ex:NotFreeButLPexplain} Let $\A$ be the arrangement in Example~\ref{ex:NotFreeButLP}. Then $\A$ is not free since there is a minimal generator of
$\rgin(J(\A))$ that is divisible by $z$. 
\end{Example}

The following Conjecture first appeared in \cite{Gin-freearr}.

\begin{Conjecture}\label{conj:generatZ}
Let $\A$ be a central arrangement in $\KK^l$, 
and consider $d_0=\min\{d~|~x_2^{d}\in\rgin(J(\A))\}$.
If $\rgin(J(\A))$ has a minimal generator $t$ that involves the third variable of $S$, then $\deg(t)\ge d_0$.
\end{Conjecture}


\section{Hyperplane arrangements and $k$-Lefschetz properties}

In this section, we study the Jacobian algebra $S/J(\A)$ of an arrangement $\A$ from the point of view of the $k$-Lefschetz properties.

Directly from Theorem~\ref{theo:SLP2dim}, we obtain the following result for arrangements in $2$-dimensional space.
\begin{Lemma}\label{lemma:l=2SLP}
Let $\A$ be a central arrangement in $\KK^2$. Then $S/J(\A)$ has the $2$-SLP.
\end{Lemma}
%
%
%
The freeness of an arrangement $\A$ forces their Jacobian algebra $S/J(\A)$ to have the $l$-SLP.
\begin{Theorem}\label{thm:FreelSLP}
Let $\A$ be a free arrangement in $\KK^l$. Then $S/J(\A)$ has the $l$-SLP.
\end{Theorem}
\begin{proof}
If $l=2$ we can directly conclude by Lemma~\ref{lemma:l=2SLP}.
Assume $l\geq3$. By Theorem~\ref{theo:firstequivfregin}, $\rgin(J(\A)) = \ideal{x_1^{n-1},\; x_1^{n-2}x_2^{\lambda_1},\; \dots,\; x_2^{\lambda_{n-1}}}$. This implies that $\rgin(J(\A))$ is an almost revlex ideal. By Theorem~\ref{theo:almrevlexlSLP}, 
$S/\rgin(J(\A))$ has the $l$-SLP. By Theorem~\ref{theo:kSLPtorgin}, $S/J(\A)$ has the $l$-SLP.
\end{proof}
Notice that Theorem~\ref{thm:FreelSLP} is not an equivalence.

\begin{Example}
Let $\A$ be the arrangement in $\RR^3$ of Example~\ref{ex:NotFreeButLP}. As described in Example~\ref{ex:NotFreeButLPexplain},
$\A$ is non-free. However, a direct computation shows that $z$ is a strong Lefschetz element for $S/\rgin(J(\A))$.
This implies that $S/\rgin(J(\A))$ has the $1$-SLP, and hence, by Corollary~\ref{cor:l-2slptolslp}, it has the $3$-SLP.
By Theorem~\ref{theo:kSLPtorgin}, also $S/J(\A)$ has the $3$-SLP.
\end{Example} 

Not all arrangements have their Jacobian algebra that has the $l$-WLP.

\begin{Example}\label{ex:AnotWLPinR4}
Let $\A$ be the arrangement in $\RR^4$ with defining polynomial $Q(\A)=xyzw(x-y+z)(y+z-3w)(x+z+w)(x-5w)$.
In this case we have that $\HF(S/\rgin(J(\A)),9)=180$ and $\HF(S/\rgin(J(\A)),10)=207$.
This shows that the multiplication by $w$ from $(S/\rgin(J(\A)))_9$ to $(S/\rgin(J(\A)))_{10}$ is not surjective.
On the other hand, $x^2y^5z^2w$ is a minimal generator of $\rgin(J(\A))$ but $x^2y^5z^2\notin\rgin(J(\A))$, 
and hence the multiplication by $w$ from degree $9$ to degree $10$ is not even injective.
This shows that $w$ is not a Lefschetz element for $S/\rgin(J(\A))$, and hence, by Proposition~\ref{prop:kwlplastvarok},
$S/\rgin(J(\A))$ does not have the $1$-WLP.
By Theorem~\ref{theo:kSLPtorgin}, also $S/J(\A)$ does not have the $1$-WLP.
\end{Example}

%
%
%
%
%
%
If Conjecture~\ref{conj:generatZ} holds, this would give us informations on the Jacobian algebra of arrangements in $\KK^3$. 

\begin{Proposition}
Let $\A$ be a central arrangement in $\KK^3$. 
If Conjecture~\ref{conj:generatZ} holds, then $S/J(\A)$ has the $3$-WLP.
\end{Proposition}
\begin{proof}
By Proposition~8.8 from \cite{palezzato2019lefschetz}, if Conjecture~\ref{conj:generatZ} holds, then $S/J(\A)$ has the $1$-WLP.
We conclude by Corollary~\ref{cor:l-2slptolslp}.
\end{proof}

\bigskip
\paragraph{\textbf{Acknowledgements}} 
During the preparation of this article the second author was supported by JSPS Grant-in-Aid for Early-Career Scientists (19K14493).




\bibliography{bibliothesis}{}

\begin{thebibliography}{10}

\bibitem{bayer1987criterion}
D.~Bayer and M.~Stillman.
\newblock A criterion for detecting m-regularity.
\newblock {\em Inventiones mathematicae}, 87(1):1--11, 1987.

\bibitem{bertone2019almost}
C.~Bertone and F.~Cioffi.
\newblock On almost revlex ideals with {H}ilbert function of complete
  intersections.
\newblock {\em Ricerche di Matematica}, pages 1--23, 2019.

\bibitem{BPT2016}
A.~M. Bigatti, E.~Palezzato, and M.~Torielli.
\newblock Extremal behavior in sectional matrices.
\newblock {\em Journal of {A}lgebra and its {A}pplications}, 18(3), 2019.

\bibitem{Gin-freearr}
A.~M. Bigatti, E.~Palezzato, and M.~Torielli.
\newblock New characterizations of freeness for hyperplane arrangements.
\newblock {\em Journal of Algebraic Combinatorics}, 51(2):297--315, 2020.

\bibitem{conca2003reduction}
A.~Conca.
\newblock Reduction numbers and initial ideals.
\newblock {\em Proceedings of the {A}merican {M}athematical {S}ociety},
  131(4):1015--1020, 2003.

\bibitem{galligo1974propos}
A.~Galligo.
\newblock A propos du th{\'e}oreme de pr{\'e}paration de {W}eierstrass.
\newblock In {\em Fonctions de plusieurs variables complexes}, pages 543--579.
  Springer, 1974.

\bibitem{harima2013lefschetz}
T.~Harima, T.~Maeno, H.~Morita, Y.~Numata, A.~Wachi, and J.~Watanabe.
\newblock {\em The {L}efschetz properties}, volume 2080 of {\em Lecture {N}otes
  in {M}athematics}.
\newblock Springer, 2013.

\bibitem{harima2009generic}
T.~Harima and A.~Wachi.
\newblock Generic initial ideals, graded {B}etti numbers, and k-{L}efschetz
  properties.
\newblock {\em Communications in Algebra}, 37(11):4012--4025, 2009.

\bibitem{migliore2013survey}
J.C. Migliore and U.~Nagel.
\newblock Survey article: a tour of the weak and strong {L}efschetz properties.
\newblock {\em Journal of Commutative Algebra}, 5(3):329--358, 2013.

\bibitem{palezzato2018free}
E.~Palezzato and M.~Torielli.
\newblock Free hyperplane arrangements over arbitrary fields.
\newblock {\em To appear in the {J}ournal of {A}lgebraic {C}ombinatorics},
  2019.

\bibitem{palezzato2019lefschetz}
E.~Palezzato and M.~Torielli.
\newblock Lefschetz properties and hyperplane arrangements.
\newblock {\em To appear in the {J}ournal of {A}lgebra}, 2019.

\bibitem{terao1980arrangementsI}
H.~Terao.
\newblock Arrangements of hyperplanes and their freeness {I}.
\newblock {\em J. Fac. Sci. Univ. Tokyo Sect. IA Math.}, 27(2):293--312, 1980.

\end{thebibliography}
\bibliographystyle{plain}

\end{document}